\documentclass[a4paper,10pt]{article}
\usepackage[utf8]{inputenc}
\usepackage[english]{babel}  
\usepackage[T1]{fontenc}      
\usepackage{amsfonts}
\usepackage{amsmath}
\usepackage{amssymb}
\usepackage{amsthm}

\usepackage{esint}	

\usepackage[pdftex]{graphicx}
\usepackage{graphics}
\usepackage{float}
\usepackage{color}
\usepackage{tikz}

\newtheorem{theorem}{Theorem}[section]

\newtheorem{proposition}[theorem]{Proposition}
\newtheorem*{proposition*}{Proposition}

\newtheorem*{theorem*}{Theorem}
\newtheorem*{lemma*}{Lemma}
\newtheorem*{corollary*}{Corollary}
\newtheoremstyle{TheoremNum}
      {\topsep}{\topsep}              		
      {\itshape}                      		
      {}                              		
      {\bfseries}                     		
      {.}                             		
      { }                             		
      {\thmname{#1}\thmnote{ \bfseries #3}}	
\theoremstyle{TheoremNum}

\theoremstyle{remark}

\title{Decomposition and pointwise estimates of periodic {G}reen functions of some elliptic equations with periodic oscillatory coefficients}
\author{Marc Josien \thanks{CERMICS, \'Ecole des Ponts, 6 et 8 avenue Blaise Pascal, 77455 Marne-La-Vall\'ee Cedex 2, (email: marc.josien@enpc.fr)}}
\begin{document}

\maketitle
\begin{abstract}
    This article is about the $\mathbb{Z}^d$-periodic Green function~$G_n(x,y)$ of the multiscale elliptic operator~$Lu=-{\rm div}\left( A(n\cdot) \cdot \nabla u \right)$, where~$A(x)$ is a $\mathbb{Z}^d$-periodic, coercive, and H\"older continuous matrix, and~$n$ is a large integer.
    We prove here pointwise estimates on~$G_n(x,y)$, $\nabla_x G_n(x,y)$, $\nabla_y G_n(x,y)$ and~ $\nabla_x \nabla_y G_n(x,y)$ in dimensions~$d \geq 2$.
    Moreover, we derive an explicit decomposition of this Green function, which is of independent interest.    
    These results also apply for systems.
\end{abstract}

\paragraph{Keywords:} Green function, periodic homogenization, multiscale problems.

\section{Introduction}
    In this article, we consider the periodic Green function~$G_n(x,y)$ associated with the multiscale problem 
    \begin{align}\label{Homog_Green_DivAGradn}
	\left\{
	    \begin{aligned}
	        &-{\rm div} \left( A(nx) \cdot \nabla u_n(x)\right)=f(x)-\int_{{\rm Q}} f && \quad\text{for}\quad x \in {\rm Q},\\
		& \int_{{\rm Q}} u_n= 0 \quad\text{and}\quad \text{$u_n$ is ${\rm Q}$-periodic},
	    \end{aligned}
	\right.
    \end{align}
    where~$n \in \mathbb{N}$ is expected to be very large and~${\rm Q}=[-1/2,1/2]^d$ is the unit cube in dimensions~$d \geq 2$.
    Hereafter, we write ``periodic'' for ``${\rm Q}$-periodic''.
    Here~$A$ satisfies the classical assumptions of ellipticity, periodicity and H\"older continuity (see~\cite{AvellanedaLin}).
    We first derive pointwise estimates for~$G_n$ and its derivatives~$\nabla_x  G_n$, $\nabla_yG_n$ and~$\nabla_x \nabla_y G_n$.
    Although these estimates are seemingly classical, we have not found them elsewhere in the literature  in the special case of periodic boundary conditions.
    In particular, we refer the reader to~\cite{LegollGreen}, which collects similar estimates, but concerning the Green function of elliptic problems with periodic coefficients in~$\mathbb{R}^d$ (and not the periodic Green function).
    We also express the periodic Green function~$G_n$ in terms of the Green function of the operator~$-{\rm div} \left( A (n \cdot) \cdot \nabla \right)$ in~$\mathbb{R}^d$.
    This latter result, which yields an alternative proof of estimates on the periodic Green function, is of independent interest. 
    Our proofs crucially make use of homogenization tools of Avellaneda and Lin \cite{AvellanedaLin,Avellaneda1991}.
    
    Our study is motivated by the fact that, in numerical applications, numerous homogenization problems are set on cubes with periodic boundary conditions (rather than on an infinite domain).
    For example, we refer the reader to~\cite{Anantharaman_2012}, where the homogenized matrix of a random medium is approximated, by appealing to a periodic problem
    (such a strategy has been recently used in~\cite{Minvielle_2015,LegollRothe}).

    \medskip
    
    Estimating the behavior of the Green function  of elliptic problems has attracted much attention, as the Green functions are a useful tool for getting estimates.
    Indeed, the solution~$u_n$ to~\eqref{Homog_Green_DivAGradn} can be written as an integral of the forcing term (\textit{e.g.}, $f$ in~\eqref{Homog_Green_DivAGradn}, which can be in the form~$f={\rm div}(H)$), multiplied by the Green function.
    Thus, if the Green function (or its derivatives) is controlled, then one can estimate the solution~$u_n$ (or its derivatives) directly from the forcing term, using the Young inequality (see, \textit{e.g.}, \cite[Chap.\ I p.\ 7-12]{BerghLofstrom}, and see~\cite{KLSGreenNeumann} for such manipulations).
    However, let us remark that, by duality, such estimates can also be used to get back to the properties of the Green function.
    We refer the reader to~\cite{KLSGreenNeumann}, which goes back and forth from properties of the Green function to estimates on the solution to the oscillating problem.

    The behavior of the Green function~$G$ of the following Dirichlet problem:
    \begin{align*}
	\left\{
	\begin{aligned}
	    &-{\rm div}\left(A(x) \cdot \nabla u(x) \right)=f(x) &&\quad\text{in}\quad \Omega,
	    \\
	    &u=0 &&\quad\text{on}\quad \partial \Omega,
	\end{aligned}
        \right.
    \end{align*}
    with elliptic and bounded matrix~$A$ has been explored in the seminal article~\cite{GruterWidman} (here~$\Omega$ is a sufficiently regular bounded domain).
    Without any regularity assumption on~$A$ (and without any hypothesis about the structure of~$A$), the authors derive optimal pointwise estimates on~$G$.
    But they need to assume that the matrix~$A$ is continuous and sufficiently regular in order to obtain pointwise estimates on the gradients~$\nabla_x G$, $\nabla_yG$, and on the second derivatives~$\nabla_x \nabla_y G$.
    Loosely speaking, they show under suitable assumptions that these quantities behave as if~$G$ was the fundamental solution to the Laplace equation (see~\cite[Chap.\ II, p.\ 13-30]{GT}), namely (in dimension~$d\geq 3$):
    \begin{align}
	\label{Homog_Green_GW0}
	&\left|G(x,y)\right|\leq C |x-y|^{-d+2},
	\\
	\label{Homog_Green_GW1}
         &\left|\nabla_x G(x,y)\right| \leq C |x-y|^{-d+1}, \qquad \left|\nabla_y G(x,y)\right| \leq C |x-y|^{-d+1},
        \\
        &\left|\nabla_x\nabla_y G(x,y) \right| \leq C  |x-y|^{-d}.
        \label{Homog_Green_GW2}
    \end{align}
    Since the domain of interest~$\Omega$ is bounded, the above quantity $|x-y|$ is bounded; hence, the difficulty in the above estimates is obviously when~$x$ is close to~$y$.
    Their results have been generalized to systems of elliptic equations. In particular, the same type of results is proved in~\cite{Dolzmann}, provided that the matrix~$A$ is sufficiently regular.
    
    On the opposite side, problems like~\eqref{Homog_Green_DivAGradn} have the specificity that the coefficients~$A(n\cdot)$ are more and more oscillating when~$n$ increases, since the caracteristic scale~$1/n$ of the microstructure goes smaller and smaller.
    Therefore, the results that rely on the regularity of the coefficient do not apply \textit{uniformly}: the constants~$C$ of the estimates~\eqref{Homog_Green_GW1} and~\eqref{Homog_Green_GW2} blows up when~$n$ goes to infinity.
    With a totally different approach than above, Avellaneda and Lin have proved that the solutions to oscillatory elliptic problems enjoy H\"older and Lipschitz regularity properties, if the matrix~$A$ is elliptic, periodic, and H\"older continuous (see~\cite{AvellanedaLin}).
    For that purpose, they introduced a so-called compactness method, showing that the oscillatory problems inherit regularity from the homogenized problem.
    Applying their results to the Green function in~$\mathbb{R}^d$ itself, they derived the same type of estimates as~\eqref{Homog_Green_GW0}, \eqref{Homog_Green_GW1}, and~\eqref{Homog_Green_GW2}.
    We refer the reader to~\cite{LegollGreen} for a review on the pointwise estimates on multiscale Green functions in~$\mathbb{R}^d$, for matrices~$A$ that are elliptic, bounded, periodic and sufficiently regular.
    
    In~\cite{Avellaneda1991}, Avellaneda and Lin described the asymptotic behavior, in the limit where the small scale vanishes, of the Green function in~$\mathbb{R}^d$ of periodic elliptic equations by using the Green function of the homogenized problem.
    Using the same techniques, the authors of~\cite{KLSGreenNeumann,KLS_2013} established the same kind of asymptotics for the Green function of the multiscale problem set in a bounded domain, for Dirichlet and Neumann boundary conditions.
    
    \medskip
    
    Periodic Green functions can sometimes be expressed thanks to the associated Green functions in the whole space~$\mathbb{R}^d$.
    For example, such a decomposition can be found in~\cite{Le_Bris_GreenPer} for the case of the Laplacian.
    This consists in a series involving the Green function in $\mathbb{R}^d$, translated on the grid $\mathbb{Z}^d$.
    The main difficulty of this decomposition is to ensure that the series actually converges; in the case of the Laplacian, this is shown by resorting to the local symmetries of the Green function of the Laplacian.
    We address the question of building a similar decomposition for the case~\eqref{Homog_Green_DivAGradn}.
    
    \medskip
    
    Most of the theoretical material and ideas used in the present article are borrowed from~\cite{AvellanedaLin,Avellaneda1991,LegollGreen,KLSGreenNeumann} for the homogenization aspect, and from~\cite{Le_Bris_GreenPer} for the decomposition of the periodic Green function.

  \subsection{Main results}
    Before getting to the oscillatory problem, we first establish the existence and the uniqueness of the periodic Green function for general periodic, elliptic and bounded coefficients.
    Henceforth, we denote by a subscript  ``${\rm per}$'' the functional spaces of periodic functions:
    for example, ${\rm{L}}^2_{{\rm per}}({\rm Q})$ is the set of functions defined on~$\mathbb{R}^d$ that are periodic and square integrable on the cube~${\rm Q}$.
    We consider the operator
      \begin{align*}
	T:f \mapsto u,
      \end{align*}
      where~$f \in {\rm{L}}^2_{{\rm per}}({\rm Q})$ and~$u$ is the unique periodic solution with zero mean to
      \begin{align}\label{Homog_Green_DivAGrad}
	-{\rm div} \left( A(x) \cdot \nabla u(x)\right)=f(x)-\int_{{\rm Q}}f \quad\text{for}\quad x \in {\rm Q},
      \end{align}
      in which the matrix~$A$ is periodic, elliptic and bounded. Namely, there exists a constant~$\mu>0$ such that~$A$ satisfies
      \begin{align}
           \label{Homog_Green_Ellipticite}
	    &\mu |\xi|^2 \leq A(x) \cdot\xi \cdot \xi \leq \mu^{-1} \left|\xi\right|^2 && \forall x, \xi \in \mathbb{R}^{d},
	    \\
	    \label{Homog_Green_Periodicite}
	    &A(x+z)=A(x) && \forall x \in \mathbb{R}^d, z \in \mathbb{Z}^d.
      \end{align}
      The operator~$T$ admits the following integral formulation, involving the so-called periodic Green function~$G$ associated with the operator~$-{\rm div}\left(A\cdot \nabla \right)$:
      \begin{align}
          Tf(x)=\int_{{\rm Q}} G(x,y) f(y) {\rm{d}} y. \label{Homog_Green_DefGreen}
      \end{align}
      By classical arguments (see, \textit{e.g.}, \cite{GruterWidman}), such a Green function~$G$ exists and is unique (see Section~\ref{Homog_Green_SecExiste} for a precise statement).
      It satisfies the following equation:
      \begin{align}
	  \label{Homog_Green_DefG2}
	  -{\rm div}_x \left(A(x) \cdot \nabla_x G(x,y) \right) = \delta_y(x)-1 \quad\text{in}\quad {\rm Q}.
      \end{align}
      
    \medskip
    
    Using a method that can be found in~\cite[Th.\ 13]{AvellanedaLin} (see also the proof of \cite[Th.\ 3.3]{KLSGreenNeumann}), we show a pointwise estimate on the periodic Green function~$G$  associated with the operator~$-{\rm div}\left(A \cdot \nabla \right)$. 
    In dimension $d=2$, this estimate on~$G(x,y)$ is logarithmic, which introduces some technicalities.
    
    \begin{proposition}
      \label{Homog_Green_ThGreen1}
      Let the dimension be~$d \geq 2$.
      Assume that~$A\in {\rm{L}}^\infty_{{\rm per}}\left({\rm Q},\mathbb{R}^{d^2}\right)$ satisfies~\eqref{Homog_Green_Ellipticite} and~\eqref{Homog_Green_Periodicite}.
      Let~$G$ be the periodic Green function associated with the operator~$-{\rm div}\left(A \cdot \nabla \right)$.
      Then there exists a constant~$C>0$ that only depends on~$d$ and~$\mu$ such that the following estimates are satisfied, for all~$x \in \mathbb{R}^d$ and~$y \in x+ {\rm Q}$,  with~$x \neq y$:
      \begin{align}
	\label{Homog_Green_Green1_debut}
	  &\text{ if } d \geq 3,
	  &&\left|G(x,y) \right| \leq C |x-y|^{-d+2},
	  \\
	\label{Homog_Green_Green1_2D_debut}
	  &\text{ if } d =2,
	  &&\left|G(x,y) \right| \leq C \log(2+|x-y|).
      \end{align}
    \end{proposition}
    The proof of Proposition~\ref{Homog_Green_ThGreen1} is postponed until Section~\ref{Homog_Green_SecGreen1}.

    One can apply the above result to the multiscale problem~\eqref{Homog_Green_DivAGradn}.
    Thus, if~$A\in {\rm{L}}^\infty_{{\rm per}}\left({\rm Q},\mathbb{R}^{d^2}\right)$ satisfies~\eqref{Homog_Green_Ellipticite} and~\eqref{Homog_Green_Periodicite}, then the periodic Green function~$G_n$  associated with the operator~$-{\rm div}\left(A(n\cdot) \cdot \nabla \right)$ satisfies~\eqref{Homog_Green_Green1_debut} or~\eqref{Homog_Green_Green1_2D_debut} (depending on the dimension~$d$), for a constant $C$ that only depends on~$d$ and~$\mu$ (and not on~$n$).

    Now, we consider a matrix~$A$ that is elliptic, periodic, and also H\"older continuous:
    \begin{align}
	   \label{Homog_Green_Holder}
	    &A \in {\rm{C}}^{0,\alpha}\left({\rm Q},\mathbb{R}^{d^2}\right),
    \end{align}
    for $\alpha \in (0,1)$.
    Using the results of~\cite{AvellanedaLin}, we derive pointwise estimates on the  gradients~$\nabla_x G_n$ and~$\nabla_y G_n$ and on the second derivatives~$\nabla_x \nabla_y G_n$ of the periodic Green function~$G_n$  associated with the operator~$-{\rm div}\left(A(n\cdot) \cdot \nabla \right)$.
    
    \begin{proposition}
	\label{Homog_Green_ThGreen2}
	Let the dimension be~$d \geq 2$.
	Assume that~$A\in {\rm{L}}^\infty_{{\rm per}}\left({\rm Q},\mathbb{R}^{d^2}\right)$ satisfies~\eqref{Homog_Green_Ellipticite}, \eqref{Homog_Green_Periodicite} and~\eqref{Homog_Green_Holder}.
	Let~$G_n$ be the periodic Green function associated with the operator~$-{\rm div}\left(A(n\cdot) \cdot \nabla \right)$.
	Then, there exists a constant~$C>0$ such that, for all~$n \in \mathbb{N} \backslash \{0\}$, $x \in \mathbb{R}^d$ and~$y \in x+ {\rm Q}$,  with~$x \neq y$,
	\begin{align}
	  \label{Homog_Green_Green2}
	    &\left|\nabla_x G_n(x,y) \right| \leq C |x-y|^{-d+1},
	    \\
	  \label{Homog_Green_Green2prim}
	    &\left|\nabla_y G_n(x,y) \right| \leq C |x-y|^{-d+1},
	    \\
	  \label{Homog_Green_Green3}
	    &\left|\nabla_x  \nabla_y G_n(x,y) \right| \leq C |x-y|^{-d}.
	\end{align}
    \end{proposition}
    The proof of Proposition~\ref{Homog_Green_ThGreen2} is postponed until Section~\ref{Homog_Green_SecGreen2}.
    
    Let us underline that the salient point of Proposition~\ref{Homog_Green_ThGreen2} is that the constant~$C$ does \textit{not} depend on the characteristic scale~$1/n$ of the microstructure.
    The latter estimates are not unexpected; see, \textit{e.g.}, \cite[Prop.\ 8]{LegollGreen} for similar estimates on the Green function in the whole space~$\mathbb{R}^d$.
    
    On the first hand, as is shown in~\cite[Th.\ 1.1]{GruterWidman}, in the case of Dirichlet boundary conditions, Estimate~\eqref{Homog_Green_Green1_debut} does not require any regularity assumption on~$A$.
    As expected, it is also the case for periodic boundary conditions.
    On the other hand, Estimates~\eqref{Homog_Green_Green2}, \eqref{Homog_Green_Green2prim} and~\eqref{Homog_Green_Green3} critically rely on the fact that~$A$ is periodic and sufficiently regular.
  
    \medskip
    Using another approach, reminiscent of~\cite[p.\ 130-131]{Le_Bris_GreenPer}, we show a decomposition for the periodic Green function~$G$.
    This formula extensively uses the corresponding Green function $\mathcal{G}$ in~$\mathbb{R}^d$ of the operator~$-{\rm div}\left(A \cdot \nabla\right)$, which satisfies
    \begin{align}
      \label{Homog_Green_DefGRd}
      -{\rm div}\left(A(x) \cdot \nabla_x \mathcal{G}(x,y) \right)=\delta_y(x) \quad\text{in}\quad \mathbb{R}^d.
    \end{align}
    
    \begin{proposition}
      \label{Homog_Green_ThDecompose}
      Let the dimension be~$d \geq 3$.
      Assume that~$A \in {\rm{L}}^\infty_{{\rm per}}\left({\rm Q},\mathbb{R}^{d^2}\right)$ satisfies~\eqref{Homog_Green_Ellipticite}, \eqref{Homog_Green_Periodicite} and~\eqref{Homog_Green_Holder}.
      Let~$G$ be the periodic Green function associated with the operator~$-{\rm div}\left(A \cdot \nabla \right)$.
      Then, the function~$G$ can be decomposed as
      \begin{align}\label{Homog_Green_Decomposition}
	G(x,y)
	=\sum_{m=0}^{+\infty} 
	\left(\sum_{k \in \Gamma_m} 
	H^k(x,y) \right),
      \end{align}
      where the functions~$H^k$ are defined by 
      \begin{align}
	\nonumber
	H^k(x,y)
	&:=
	\mathcal{G}(x,y-k)
	-\int_{{\rm Q}} \mathcal{G}(x,y+y'-k) {\rm{d}} y'-\int_{{\rm Q}} \mathcal{G}(x+x',y-k) {\rm{d}} x'
	\\
	&~~~~+\int_{{\rm Q}} \int_{{\rm Q}} \mathcal{G}(x+x',y+y'-k) {\rm{d}} y' {\rm{d}} x',
	\label{Homog_Green_DefHk}
      \end{align}
      and the function $\mathcal{G}$ by~\eqref{Homog_Green_DefGRd}, and the sets~$\Gamma_m$ by
      \begin{align}
	\label{Homog_Green_DefSet1}
        &\Gamma_0=\left\{k \in \mathbb{Z}^d, 0 \leq k\cdot \left(A^\star_{\rm s}\right)^{-1} \cdot k < 2^2 \right\}, \\
        \label{Homog_Green_DefSet2}
        &\Gamma_m =\left\{k \in \mathbb{Z}^d, 2^{2m} \leq k\cdot \left(A^\star_{\rm s}\right)^{-1} \cdot k < 2^{2m+2} \right\} && \quad\text{if}\quad m \geq 1,
      \end{align}
      where~$A^\star_{\rm s}$ is the symmetric part of the homogenized matrix~$A^\star$ associated with the matrix~$A$.
    \end{proposition}
    The proof of Proposition~\ref{Homog_Green_ThDecompose} is postponed until Section~\ref{Homog_Green_SecAnother}.

    The above decomposition~\eqref{Homog_Green_Decomposition} naturally appears as a reasonable candidate, being close (but not equivalent) to the decomposition~\cite[p.\ 130-131]{Le_Bris_GreenPer}.      
    But the difficulty is to ensure that the series actually converges, in the sense that 
    \begin{align}\label{Sommable}
      \sum_{m=0}^{+\infty} 
	\left|\sum_{k \in \Gamma_m} 
	H^k(x,y) \right| < +\infty \quad\text{for}\quad x \neq y.
    \end{align}
    In~\cite[p.\ 130-131]{Le_Bris_GreenPer}, where the Laplacian with periodic boundary conditions is studied, the convergence is obtained by appealing to the \textit{local} symmetries of the Green function of the Laplacian.
    This cannot be applied to our case.
    Here, the convergence is a consequence of the \textit{long-range} behavior of the Green function~$\mathcal{G}$.
    Thanks to the periodicity of~$A$, the function~$\mathcal{G}$ can be efficiently approximated at large scale by the Green function of the homogenized problem (see~\cite{Avellaneda1991,KLSGreenNeumann}). 
    Hence, taking advantage of the long-range symmetries of the Green function of the homogenized problem, one can prove the convergence of the series in~\eqref{Homog_Green_Decomposition}.
    In this regard, we underline that, in general, the series~\eqref{Homog_Green_Decomposition} does not converge absolutely with respect to~$k$:
    \begin{align*}
      \sum_{k \in \mathbb{Z}^d} \left| H^k(x,y) \right| = +\infty \quad\text{for}\quad x \neq y.
    \end{align*}
    This fact appears as a byproduct of the proof.
    
    Last but not least, it should be underlined that  the above decomposition provides an alternative way for showing the pointwise estimates~\eqref{Homog_Green_Green1_debut} on the multiscale periodic Green function~$G_n$ of the operator~$-{\rm div}\left(A(n\cdot)\cdot\nabla\right)$.
    Indeed, the proof of Proposition~\ref{Homog_Green_ThDecompose} (in Section~\ref{Homog_Green_SecAnother}) implies that the series~\eqref{Homog_Green_Decomposition} converges uniformly with respect to~$n$.

  \subsection{Extension  to systems}\label{Sec_Systems}
    Our proof of the existence and the uniqueness of the Green function uses the De Giorgi-Nash-Moser theorem. In dimension $d \geq 3$, this ingredient can be replaced by the~${\rm{W}}^{1,p}$ and~${\rm{L}}^{\infty}$ estimates in~\cite[Lem.\ 2 \& Lem. 3]{Dolzmann}.
    Hence, there also exists a unique periodic Green function of the operator $Lu:=\left(L^\alpha u\right)_{\alpha \in [\![1,m]\!]}$ defined by
    \begin{align*}
      L^\alpha u:= - {\rm div} \left( \sum_{\beta=1}^m A^{\alpha\beta} \cdot \nabla u^\beta \right)
      =- \sum_{i, j=1}^d \partial_i \left( \sum_{\beta=1}^m A_{ij}^{\alpha\beta} \partial_j u^\beta \right),
    \end{align*}
    where~$A=\left( A_{ij}^{\alpha\beta}\right)$, for~$i, j \in [\![1,d]\!]$ and~$\alpha, \beta \in [\![1,m]\!]$, $m \in \mathbb{N}$, is \textit{continuous}, periodic, and elliptic in the following sense:
    \begin{align}
      \label{Homog_Green_EllipSys}
      \mu |\xi|^2 \leq \sum_{i, j=1}^d \sum_{\alpha, \beta=1}^m A_{ij}^{\alpha\beta}(x) \xi^{\alpha}_i \xi^{\beta}_j \leq \mu^{-1} \left|\xi\right|^2 && \forall x \in \mathbb{R}^d, \xi=\left(\xi_i^\alpha\right) \in \mathbb{R}^{dm},
    \end{align}
    In this case, the periodic Green function~$G$ (which is a matrix) satisfies, for all~$\alpha, \gamma \in [\![1,m]\!]$,
    \begin{align}
      \label{Homog_Green_DefG}
      -{\rm div}_x \left(\sum_{\beta=1}^m A^{\alpha\beta}(x) \cdot \nabla_x G^{\beta\gamma}(x,y) \right) = \delta^{\alpha\gamma} \left(\delta_y(x)-1\right) \quad\text{in}\quad {\rm Q},
    \end{align}
    where~$\delta^{\alpha\beta}$ is the Kronecker symbol.
    
    As can be seen in Sections~\ref{Homog_Green_SecGreen1} and~\ref{Homog_Green_SecGreen2}, the proofs of Propositions~\ref{Homog_Green_ThGreen1} and~\ref{Homog_Green_ThGreen2} involve arguments that are also valid if we study periodic oscillatory systems instead of equations (note that the seminal article~\cite{AvellanedaLin} dealt with systems).
    More precisely, if  $d \geq 3$, the periodic Green function~$G_n$ associated with the operator~$-{\rm div}\left(A(n\cdot)\cdot \nabla \right)$ satisfies~\eqref{Homog_Green_Green1_debut}, \eqref{Homog_Green_Green1_2D_debut}, \eqref{Homog_Green_Green2}, \eqref{Homog_Green_Green2prim}, and~\eqref{Homog_Green_Green3}, provided that~$A=\left(A_{ij}^{\alpha\beta}\right)$ is periodic, satisfies~\eqref{Homog_Green_EllipSys}, and is H\"older continuous.
    Notably, the H\"older estimate~\cite[Lem.\ 9]{AvellanedaLin} can be used instead of the De Giorgi-Nash-Moser theorem in the proof of Proposition~\ref{Homog_Green_ThGreen1}.
    Besides, the Lipschitz estimate borrowed from~\cite[Lem.\ 16]{AvellanedaLin} that we use in the proof of Proposition~\ref{Homog_Green_ThGreen2} also applies.
    
    Finally, the formula~\eqref{Homog_Green_Decomposition} can also be generalized to the case of systems, using appropriate sets~$\Gamma_m^{\alpha,\beta}$ while decomposing~$G^{\alpha\beta}$ (see Section~\ref{Homog_Green_Systems}).
    
  \subsection{Outline}
    Our article is articulated as follows.
    In Section~\ref{Homog_Green_SecExiste}, we precisely explain in which sense there exists a unique periodic Green function.
    We briefly justify this fact by classical arguments.
    Next, in Section~\ref{Homog_Green_SecGreen1}, we proceed with the proof of Proposition~\ref{Homog_Green_ThGreen1}.
    In dimension $d \geq 3$, the proof is based on a duality argument involving the De Giorgi-Nash-Moser theorem.
    In dimension $d=2$, using a trick from~\cite{AvellanedaLin}, it reduces to expressing the~$2$-dimensional periodic Green function as the integral of a~$3$-dimensional periodic Green function.
    In Section~\ref{Homog_Green_SecGreen2}, combining Estimate~\eqref{Homog_Green_Green1_debut} of Proposition~\ref{Homog_Green_ThGreen1} and the Lipschitz estimates of~\cite{AvellanedaLin}, we show~\eqref{Homog_Green_Green2} (and similarly~\eqref{Homog_Green_Green2prim}), from which we deduce~\eqref{Homog_Green_Green3}.
    Finally, in Section~\ref{Homog_Green_SecAnother}, we prove Proposition~\ref{Homog_Green_ThDecompose}, which, under suitable hypotheses, yields a decomposition for the periodic Green function.
    For the sake of simplicity, we first study the case where the homogenized matrix is the identity in Section~\ref{Homog_Green_SecId}, and then the general case in Section~\ref{Homog_Green_Sym_Mat}.
    Additional materials about such a decomposition in the case of systems can be found in the Section~\ref{Homog_Green_Systems}.
    
\section{Existence, uniqueness and basic properties of the Green function}
\label{Homog_Green_SecExiste}

    In this section, we justify that there exists a unique periodic Green function~$G$ associated with the operator~$-{\rm div}\left(A\cdot\nabla\right)$.
      It lies in the functional space~$E$ containing all the functions~$G(x,y)$ satisfying, for all~$p \in \left[ 1, \frac{d}{d-2}\right)$ (by convention, if $d=2$, then $d/(d-2)=+\infty$) and~$q \in \left[ 1, \frac{d}{d-1}\right)$, 
      \begin{align}
	  \label{Homog_Green_1}
	  &\sup_{y \in {\rm Q}} \left\|G(\cdot,y)\right\|_{{\rm{L}}^{p}({\rm Q})}<+\infty
	  &&\sup_{y \in {\rm Q}} \left\|\nabla_x G(\cdot,y)\right\|_{{\rm{L}}^{q}\left({\rm Q}\right)} <+\infty,
	  \\
	  \label{Homog_Green_3}
	  &\sup_{x \in {\rm Q}} \left\|G(x,\cdot)\right\|_{{\rm{L}}^{p}({\rm Q})} <+\infty,
	  &&\sup_{x \in {\rm Q}} \left\|\nabla_y G(x,\cdot)\right\|_{{\rm{L}}^{q}\left({\rm Q}\right)} <+\infty.
      \end{align}
      
    \begin{proposition}
	\label{Homog_Green_PropExistUne}
	Let the dimension be~$d \geq 2$.
	Assume that~$A \in {\rm{L}}^\infty_{{\rm per}}\left({\rm Q},\mathbb{R}^{d^2}\right)$ satisfies~\eqref{Homog_Green_Ellipticite} and~\eqref{Homog_Green_Periodicite}.
	Then there exists a unique periodic Green function~$G(x,y)$  associated with the operator~$-{\rm div}\left(A\cdot \nabla \right)$ -namely, $G$ satisfies~\eqref{Homog_Green_DefGreen}- that is in the space~$E$, defined by~\eqref{Homog_Green_1} and~\eqref{Homog_Green_3}.
	Moreover, the function~$G^\dagger(x,y):=G(y,x)$ is the periodic Green function associated with the operator~$-{\rm div}\left(A^T\cdot \nabla \right)$.
	Last, $G$ is the unique periodic solution in~$E$ to~\eqref{Homog_Green_DefG2} satisfying
	\begin{align}
	    \label{Homog_Green_MoyenneGreen}
	    \int_{{\rm Q}} G(x,y) {\rm{d}} y = 0 &&  \forall x \in {\rm Q},
	    \\
	    \label{Homog_Green_MoyenneGreen2}
	    \quad\text{and}\quad
	    \int_{{\rm Q}} G(x,y) {\rm{d}} x = 0 &&  \forall y \in {\rm Q}.
	\end{align}
    \end{proposition}
    
      Actually, if~$d\geq 3$, the Green function is expected to satisfy the following estimates:
      \begin{align}
	  \label{Homog_Green_Marcinkiewicz1}
          &\sup_{y \in {\rm Q}} \left\|G(\cdot,y)\right\|_{{\rm{L}}^{\frac{d}{d-2},\infty}({\rm Q})} <+\infty ,
          &&\sup_{y \in {\rm Q}} \left\|\nabla_x G(\cdot,y)\right\|_{{\rm{L}}^{\frac{d}{d-1},\infty}\left({\rm Q}\right)} <+\infty ,
          \\
	  \label{Homog_Green_Marcinkiewicz2}
          &\sup_{x \in {\rm Q}} \left\|G(x,\cdot)\right\|_{{\rm{L}}^{\frac{d}{d-2},\infty}({\rm Q})} <+\infty ,
          &&\sup_{x \in {\rm Q}} \left\|\nabla_y G(x,\cdot)\right\|_{{\rm{L}}^{\frac{d}{d-1},\infty}\left({\rm Q}\right)} <+\infty,
      \end{align}
      since it is the case when homogeneous Dirichlet boundary conditions are considered (see~\cite[Th.\ 1.1]{GruterWidman}).
      Here~${\rm{L}}^{p,\infty}$ denote the Marcinkiewicz spaces (see~\cite[Chap.\ I p.\ 7-11]{BerghLofstrom} for a reference on such functional spaces).
      Thus, the proposed space $E$ is not optimal.
      However, for the purpose of the present article, it not useful to find the optimal function space, since~\eqref{Homog_Green_Marcinkiewicz1} and~\eqref{Homog_Green_Marcinkiewicz2} are a straightforward corollary of Propositions~\ref{Homog_Green_ThGreen1} and~\ref{Homog_Green_ThGreen2} below.

    \medskip
    Proposition~\ref{Homog_Green_PropExistUne} can be proven by standard arguments (see, \textit{e.g.}, \cite{GruterWidman}) that we briefly describe here (see~\cite[Chapter 3]{PhD_Marc}).
    
    First, we establish existence and uniqueness by using the regularizing properties of~\eqref{Homog_Green_DivAGrad}.
    Then, we show that~$(x,y) \mapsto G(y,x)$ is the Green function of the transposed problem by considering the adjoint operator of~$T$.
    In a third step, we use the variational formulation of~\eqref{Homog_Green_DivAGrad} and establish that the Green function~$G$ satisfies~\eqref{Homog_Green_DefG2}, \eqref{Homog_Green_MoyenneGreen} and~\eqref{Homog_Green_MoyenneGreen2}.
    Finally, we show the uniqueness of the solution to~\eqref{Homog_Green_DefG2}, \eqref{Homog_Green_MoyenneGreen} and~\eqref{Homog_Green_MoyenneGreen2}, using a variational argument.
    The detailed proof of each above argument can be found in~\cite[Chapter 3]{PhD_Marc}.

\section{Pointwise estimates on the periodic Green function}\label{Homog_Green_SecGreen1}

  This section is devoted to the proof of Proposition~\ref{Homog_Green_ThGreen1}.
  For technical reasons, we proceed first with the case of dimension~$d\geq 3$, and then with the case of dimension~$d=2$.
  
  \subsection{The case of $ d\geq 3$}
  
  Pointwise estimates on the Green functions of elliptic problems with Dirichlet boundary conditions have been established in the seminal article~\cite{GruterWidman} of Gr\"uter and Widman.
  Their proof makes use of the comparison principle.
  The latter is an appropriate tool for an elliptic \textit{equation} with homogeneous Dirichlet boundary conditions: in this case, the Green function is positive.
  But, such an argument fails when considering the periodic Green function, the sign of which varies (it has zero mean).
  As a consequence, here, we resort to a duality argument and to the De Giorgi-Nash-Moser theorem (see~\cite[Th.\ 13]{AvellanedaLin}).
  Also, when considering multiscale periodic elliptic \textit{systems}, the latter theorem does not hold and the H\"older estimates of~\cite{AvellanedaLin} are necessary for concluding the proof (see Section \ref{Sec_Systems}).
  
  The proof below is a straightforward adaptation of the proof of~\cite[Th.\ 13]{AvellanedaLin}.
  The fact that we study periodic boundary conditions do not raise substantial difficulties, since the strategy involves local estimates.
  
  \medskip
  
  Let us first explain in a few words the ingredients of the proof of Proposition~\ref{Homog_Green_ThGreen1} in the case $d\geq 3$.
  The first step of the proof consists in combining the De Giorgi-Nash-Moser theorem, the classical Hilbert theory and the Sobolev injections in order to obtain an optimal~${\rm{L}}^\infty$ estimate on the periodic solution~$u$ to~\eqref{Homog_Green_DivAGrad} for localized right-hand terms~$f$.
  By a duality argument used in~\cite[Th.\ 13]{AvellanedaLin}, this provides a local ${\rm{L}}^2$ bound on the Green function $G(x,y)$.
  Using once more the De Giorgi-Nash-Moser theorem, this proves Estimate~\eqref{Homog_Green_Green1_debut}.

  \begin{proof}[Proof of Proposition~\ref{Homog_Green_ThGreen1} in dimension $d \geq 3$]
    The proof falls in two steps.
  
    Let~$x_0 \in \mathbb{R}^d, y_0 \in x_0+Q$, $x_0\neq y_0$, and~$2r:=|x_0-y_0|$ (by periodicity, it is the only relevant case).

    \begin{figure}[h]
      \begin{center}
	\includegraphics{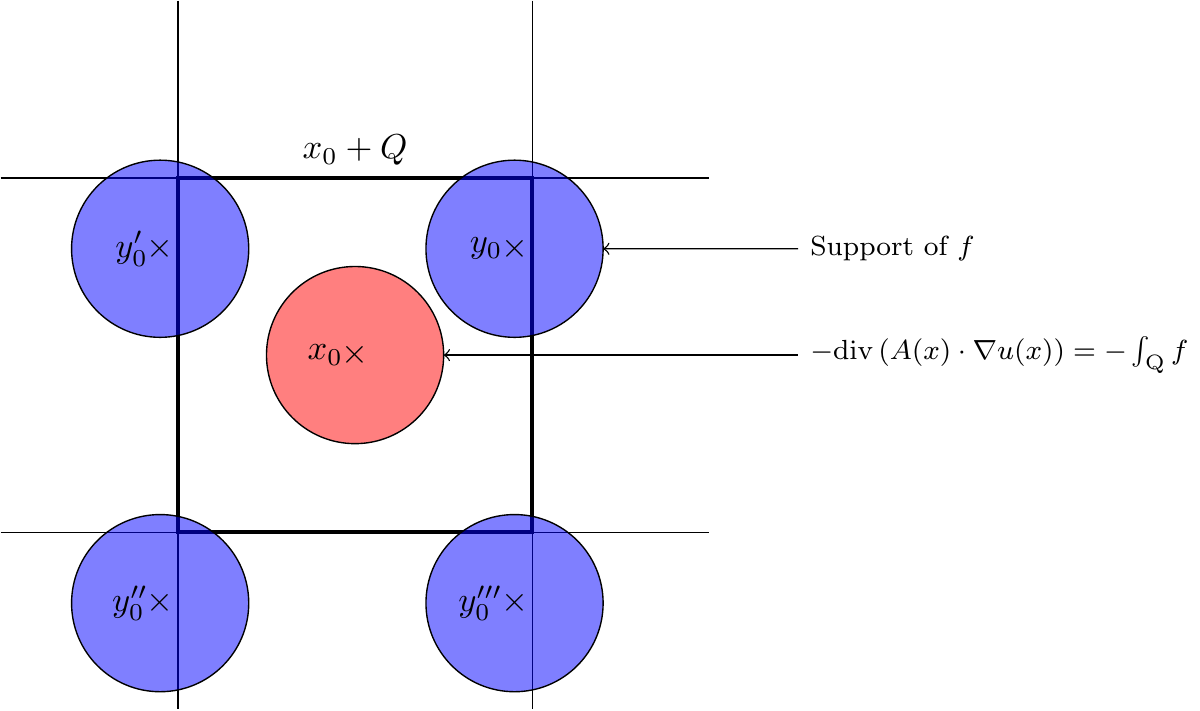}
      \end{center}
      \caption{Illustration of Step 1 of proof  of Proposition~\ref{Homog_Green_ThGreen1}.} \label{Homog_Green_Fig_Step1}
    \end{figure}
      
    \paragraph{Step 1:}
      Let us consider~$f \in {\rm{L}}_{{\rm per}}^{\frac{2d}{d+2}}\left({\rm Q}\right)$ that has a support contained in~$\mathbb{Z}^d+{\rm B}(y_0,r/4)$ (see Figure~\ref{Homog_Green_Fig_Step1}).
      Define~$u$ as the periodic solution with zero mean to~\eqref{Homog_Green_DivAGrad}. 
      Since~${\rm{L}}_{{\rm per}}^{\frac{2d}{d+2}}({\rm Q})$ is in the dual of~${\rm{H}}_{{\rm per}}^1({\rm Q})$ (see~\cite[Th.\ 9.9 p.\ 278]{Brezis}), then,  by the Lax-Milgram theorem, $u \in {\rm{H}}^1_{{\rm per}}({\rm Q})$ and there obviously holds
      \begin{align}
	  \label{Homog_Green_MajorNabla}
	  \left\|\nabla u \right\|_{{\rm{L}}^2({\rm Q})} \leq C \left\|f\right\|_{{\rm{L}}^{\frac{2d}{d+2}}\left({\rm Q}\right)}.
      \end{align}
      As the support of~$f$ is contained in~$\mathbb{Z}^d+{\rm B}(y_0,r/4)$, then~$u$ satisfies in~${\rm B}(x_0,r)$ the following equation:
      \begin{align*}
	  -{\rm div}\left( A(x)\cdot \nabla u(x) \right)=-\int_{{\rm Q}} f=-\left\{\int_{{\rm B}(y_0,r/4)} f\right\}{\rm div}\left( \frac{x-x_0}{d} \right).
      \end{align*}
      Therefore, as a consequence of the De Giorgi-Nash-Moser theorem (see~\cite[Th.\ 8.24 p.\ 202]{GT}), there exists~$\beta\in(0,1)$ and~$C>0$ depending only on~$d$ and~$\mu$ such that
      \begin{align}
	  \left|u(x_0)\right| 
	  \leq & C r^{-d/2}\left(\int_{{\rm B}(x_0,r)} \left|u(x) \right|^2{\rm{d}} x \right)^{1/2}
	  +Cr^2 \left| \int_{{\rm B}(y_0,r/4)} f \right|.
	  \label{Homog_Green_E1}
      \end{align}
      We now bound the two terms on the right-hand side.
      For the second term, by the H\"older inequality, we obtain
      \begin{align}
	  \left| \int_{{\rm B}(y_0,r/4)} f \right| \leq C r^{\frac{d-2}{2}} \left( \int_{{\rm B}(y_0,r/4)}  \left|  f(x) \right|^{\frac{2d}{d+2}} {\rm{d}} x\right)^{\frac{d+2}{2d}}.
	  \label{Homog_Green_E2}
      \end{align}
      For the first term, again using the H\"older inequality, we have
      \begin{align}
	  \left(\int_{{\rm B}(x_0,r)} \left|u(x) \right|^2{\rm{d}} x \right)^{1/2}
	  \leq & C r \left( \int_{{\rm Q}} \left|u(x)\right|^{\frac{2d}{d-2}} {\rm{d}} x  \right)^{\frac{d-2}{2d}}.
	  \label{Homog_Green_E3}
      \end{align}
      By Sobolev injection of~${\rm{H}}^1({\rm Q})$ in~${\rm{L}}^{\frac{2d}{d-2}}({\rm Q})$ (and since~$u$ has zero mean),
      \begin{align}\label{Homog_Green_E13}
	  \left( \int_{{\rm Q}} \left|u(x)\right|^{\frac{2d}{d-2}} {\rm{d}} x  \right)^{\frac{d-2}{2d}} 
	  \leq & C \left( \int_{{\rm Q}} \left| \nabla u(x)\right|^{2} {\rm{d}} x  \right)^{1/2},
      \end{align}
      whence, we deduce from~\eqref{Homog_Green_E3}, \eqref{Homog_Green_E13} and~\eqref{Homog_Green_MajorNabla} that
      \begin{align}
	  \left(\int_{{\rm B}(x_0,r)} \left|u(x) \right|^2{\rm{d}} x \right)^{1/2}
	  \leq & C r\left( \int_{{\rm B}(y_0,r/4)} \left|f(x)\right|^{\frac{2d}{d+2}} {\rm{d}} x\right)^{\frac{d+2}{2d}}.
	  \label{Homog_Green_E4}
      \end{align}
      Finally, since $2r\leq \sqrt{d}$, \eqref{Homog_Green_E1}, \eqref{Homog_Green_E2}, and~\eqref{Homog_Green_E4} yield
      \begin{align}
	  \label{Homog_Green_ulinfty}
	  \left|u(x_0)\right| \leq C r^{-(d-2)/2}\left( \int_{{\rm B}(y_0,r/4)} \left|f(x)\right|^{\frac{2d}{d+2}} {\rm{d}} x\right)^{\frac{d+2}{2d}}.
      \end{align}
      
      \paragraph{Step 2:}
      The function~$u$ can be expressed thanks to the Green function as
      \begin{align*}
	  u(x)=\int_{{\rm Q}} G(x,y) f(y) {\rm{d}} y.
      \end{align*}
      As a consequence, 
      by duality, \eqref{Homog_Green_ulinfty} implies that there exists a constant~$C>0$ such that
      \begin{align}
	  \label{Homog_Green_Glp}
	  \left( \int_{{\rm B}(y_0,r/4)} \left|G(x_0,y)\right|^{\frac{2d}{d-2}} {\rm{d}} y \right)^{\frac{d-2}{2d}} \leq C  r^{-(d-2)/2}.
      \end{align}
      As~$G(x_0,\cdot)$ satisfies
      \begin{align}\label{Homog_Green_SatisfGT}
	-{\rm div}_y \left( A^T(y)\cdot \nabla_y G(x_0,y) \right) = -1,
      \end{align}
      in~${\rm B}(y_0,r/4)$, then, using once more~\cite[Th.\ 8.24 p.\ 202]{GT},
      we obtain (in the same manner as~\eqref{Homog_Green_E1})
      \begin{align*}
	  &\left\|G(x_0,\cdot)\right\|_{{\rm{L}}^\infty\left({\rm B}(y_0,r/8)\right)} 
	  \leq r^{-d/2} \left( \int_{{\rm B}(y_0,r/4)} \left|G(x_0,y)\right|^{2} {\rm{d}} y \right)^{1/2}+Cr^2.
      \end{align*}
      By the H\"older inequality, we deduce from the above estimate that
      \begin{align*}
	&\left\|G(x_0,\cdot)\right\|_{{\rm{L}}^\infty\left({\rm B}(y_0,r/8)\right)} 
	\leq r^{\frac{2-d}{2}} \left( \int_{{\rm B}(y_0,r/4)} \left|G(x_0,y)\right|^{\frac{2d}{d-2}} {\rm{d}} y \right)^{\frac{d-2}{2d}}+Cr^2.
      \end{align*}
      Finally, since~$2r=|x_0-y_0|\leq \sqrt{d}$ and thanks to~\eqref{Homog_Green_Glp}, we deduce~\eqref{Homog_Green_Green1_debut}.
      This concludes the proof of Proposition~\ref{Homog_Green_ThGreen1} in dimension~$d\geq 3$.
  \end{proof}
  
  \subsection{The case $d=2$}\label{SecDim2}
  
    We now turn to the dimension $d=2$. 
    According to an idea that can be found in~\cite[Th.\ 13]{AvellanedaLin} (see also~\cite[Prop.\ 4]{LegollGreen}), the $2$-dimensional periodic Green function~$G$ associated with the operator~$-{\rm div}\left(A\cdot\nabla\right)$ can be expressed as
    \begin{align}
	\label{Homog_Green_3vers2}
	G(x,y) = \int_0^1 \widetilde{G}((x,t),(y,0)) {\rm{d}} t
    \end{align}
    where~$x, y \in \mathbb{R}^2$ and~$t\in \mathbb{R}$ in the equation above, and~$\widetilde{G}$ is the~$3$-dimensional periodic Green function of the following operator:
    \begin{align}
      \label{Homog_Green_DefTildeG}
      \widetilde{L}u := -{\rm div}_{x} \left( A \cdot \nabla_{x} u \right) - \partial_{tt} u.
    \end{align}
    Indeed, it is easily shown (by an integration argument) that the function~$G$ defined by~\eqref{Homog_Green_3vers2} satisfies~\eqref{Homog_Green_DefG2}, \eqref{Homog_Green_MoyenneGreen} and \eqref{Homog_Green_MoyenneGreen2}.
    Hence, $G$ is the ($2$-dimensional) periodic Green function associated with the operator~$-{\rm div}\left(A\cdot \nabla \right)$.
    
    Next, applying Proposition~\ref{Homog_Green_ThGreen1} --which we have already proved in dimension $d\geq 3$-- to the Green function $\widetilde{G}$, we obtain
    \begin{align*}
      \left|G(x,y)\right| \leq C \int_0^{1} \frac{1}{|t|+|x-y|} {\rm{d}} t \leq C \log\left(2+ |x-y| \right),
    \end{align*}
    which is~\eqref{Homog_Green_Green1_2D_debut}.
    This concludes the proof of Proposition~\ref{Homog_Green_ThGreen1} in the case~$d=2$.

\section[Pointwise estimates on the derivatives]{Pointwise estimates on the derivatives $\nabla_x G_n$, $\nabla_y G_n$ and $\nabla_x\nabla_y G_n$ of the multiscale periodic Green function}\label{Homog_Green_SecGreen2}
    \sectionmark{Pointwise estimates on the derivatives of the multiscale Green function}
    
    Proposition~\ref{Homog_Green_ThGreen2} relies on the Lipschitz theory of~\cite{AvellanedaLin}.
    Indeed, if~$u_n$ satisfies
    \begin{align}
      \label{Homog_Green_Eq_Sur_Green}
      -{\rm div}\left( A(nx) \cdot \nabla u_n(x) \right)=a \quad\text{in}\quad {\rm B}(x_0,r),
    \end{align}
    for some~$a \in \mathbb{R}$, then~\cite[Lem.\ 16]{AvellanedaLin} shows that
    \begin{align}
    \label{Homog_Green_EstimThAL}
    &\left| \nabla u_n(x_0) \right| \leq C r^{-1-d/2} \left( \int_{{\rm B}(x_0,r)} \left|u_n \right|^2 \right)^{1/2}+C|a|r,
    \end{align}
    where $C>0$ is a constant independent of $n$.
    As is detailed below, we treat separately the $2$-dimensional case.

    \begin{proof}[Proof of Proposition~\ref{Homog_Green_ThGreen2}]
      Assume first that the dimension satisfies $d\geq 3$.
      Let~$x_0 \in \mathbb{R}^d, y_0 \in x_0+Q$, $x_0\neq y_0$, and~$2r:=|x_0-y_0|$ (once more, by periodicity, this is the only relevant case).
      
      Since~$G_n(x,y_0)$ satisfies
      \begin{align*}
	-{\rm div}_x \left( A(nx)\cdot \nabla_x G_n(x,y_0) \right) = -1 \quad\text{in}\quad {\rm B}(x_0,r),
      \end{align*}
      and thanks to~\cite[Lem.\ 16]{AvellanedaLin}, \textit{id est}~\eqref{Homog_Green_EstimThAL}, there holds
      \begin{align}\label{Homog_Green_audessus}
	\left|\nabla_xG_n(x_0,y_0) \right| \leq C r^{-1} \left\| G_n(\cdot,y_0)\right\|_{{\rm{L}}^\infty({\rm B}(x_0,r))}+C r.
      \end{align}
      As a consequence, Estimate~\eqref{Homog_Green_Green1_debut} and~\eqref{Homog_Green_audessus} yield~\eqref{Homog_Green_Green2} (for~$x=x_0$ and~$y=y_0$).
      By transposition, \eqref{Homog_Green_Green2prim} is also established.
    
      By differentiating~\eqref{Homog_Green_DefG2} with respect to~$y$, we obtain
      \begin{align*}
	-{\rm div}_x \left(A(nx) \cdot \nabla_x \nabla_yG_n(x,y_0) \right) = 0 \quad\text{in}\quad {\rm B}(x_0,r/2).
      \end{align*}
      Therefore, thanks to~\eqref{Homog_Green_EstimThAL},
      \begin{align*}
	\left|\nabla_x \nabla_y G_n(x_0,y_0)\right| 
	\leq& C r^{-1} \left\| \nabla_y G(\cdot,y_0) \right\|_{{\rm{L}}^\infty\left({\rm B}(x_0,r/2)\right)}.
      \end{align*}
      By using~\eqref{Homog_Green_Green2prim} and since~$2r=|x_0-y_0|$, we conclude that
      \begin{align*}
	\left|\nabla_x \nabla_y G_n(x_0,y_0)\right| 
	\leq C |x_0-y_0|^{-d},
      \end{align*}
      and establish~\eqref{Homog_Green_Green3}.
      
      \medskip
      
      Now, let the dimension $d=2$. 
      As shown in the proof of Proposition~\ref{Homog_Green_ThGreen1}, the $2$-dimensional periodic Green function~$G_n$ can be expressed as
      \begin{align}
	  \label{Homog_Green_3vers2_n}
	  G_n(x,y) = \int_0^1 \widetilde{G}_n((x,t),(y,0)) {\rm{d}} t
      \end{align}
      where~$x, y \in \mathbb{R}^2$ and~$t\in \mathbb{R}$ in the equation above, and~$\widetilde{G}_n$ is the~$3$-dimensional periodic Green function of the operator 
      $\widetilde{L}u := -{\rm div}_{x} \left( A(n\cdot) \cdot \nabla_{x} u \right) - \partial_{tt} u$.
      Hence, we deduce the 2-dimensional versions of~\eqref{Homog_Green_Green2}, \eqref{Homog_Green_Green2prim} and~\eqref{Homog_Green_Green3} by integrating their $3$-dimensional versions applied to $\widetilde{G}_n$ (in the same manner as in Section~\ref{SecDim2}).
    \end{proof}

  \section{A decomposition of the periodic Green function}\label{Homog_Green_SecAnother}
    
    In this section, we prove Proposition~\ref{Homog_Green_ThDecompose}.
    For the sake of simplicity, we first assume that the homogenized matrix~$A^\star$ is the identity.
    We postpone the proof in the general case until Section~\ref{Homog_Green_Sym_Mat}.
    
    \subsection{Case where the homogenized matrix is the identity}\label{Homog_Green_SecId}
    
    For convenience, we have split the proof of Proposition~\ref{Homog_Green_ThDecompose} in two parts: first, we show that the series in~\eqref{Homog_Green_Decomposition} actually converges; second, we check that its limit is the periodic Green function~$G$.
    
    The two main steps of the proof of convergence are the following: first a Taylor expansion allows for expressing the terms~$H^k$ in~\eqref{Homog_Green_Decomposition} as functions of~$\nabla_x \nabla_y \mathcal{G}$. Then, we approximate the Green function of the multiscale problem with the Green function of the homogenized problem (see~\cite{Avellaneda1991}).
    Second, we take advantage of the long-range symmetries of the Green function of the homogenized problem and establish the convergence of the series in~\eqref{Homog_Green_Decomposition}. There, the sets~$\Gamma_m$ are crucial, since the convergence in~\eqref{Homog_Green_Decomposition} is not uniform in~$k$.
    
    \begin{proof}[Proof of convergence of the series in~\eqref{Homog_Green_Decomposition}]
      We denote by~$\mathcal{G}_\star(x,y)$ the fundamental solution in~$\mathbb{R}^d$ to the homogenized problem.
      The homogenized matrix is the identity; therefore~$\mathcal{G}_\star$ is explicitly expressed as
      \begin{align}\label{Homog_Green_GreenLaplace}
	\mathcal{G}_\star(x,y)=C_d |x-y|^{2-d},
      \end{align}
      where~$C_d$ is a constant (see~\cite[(2.12) p.\ 17]{GT}).
      Since~$\mathcal{G}_\star(x,y)$ only depends on~$|x-y|$, we henceforth redefine
      \begin{align*}
        \mathcal{G}_\star(x):=\mathcal{G}_\star(x,0).
      \end{align*}
      
      \paragraph{Step 1:}
	Assume that~$x \in {\rm Q}, y-x \in {\rm Q}$ and~$k \notin 4 {\rm Q}$. We reformulate~$H^k$ using the Taylor formula:
	\begin{align}
	  \nonumber
	  H^k(x,y)
	  =&
	  -\int_{{\rm Q}} x' \cdot \int_0^1 \Bigg(
	  \nabla_x \mathcal{G}(x+tx',y-k)
	  \\
	  \nonumber
	  &~~~~-\int_{{\rm Q}} \nabla_x \mathcal{G}(x+tx',y+y'-k) {\rm{d}} y' \Bigg) {\rm{d}} t {\rm{d}} x'
	  \\
	  \nonumber
	  =& \int_{{\rm Q}} x' \cdot \Bigg(\int_0^1 \int_{{\rm Q}} y'
	  \\
	  &~~~~
	   \cdot \int_0^1\nabla_y \nabla_x \mathcal{G}(x+tx',y+\tau y'-k) {\rm{d}} \tau {\rm{d}} y' {\rm{d}} t\Bigg) {\rm{d}} x'.
	  \label{Homog_Green_ExpressionDerive2}
	\end{align}
	We approximate~$\nabla_x \nabla_y \mathcal{G}$ by~$\nabla_x\nabla_y \mathcal{G}_\star$. 
	More precisely, thanks to~\cite[Corollary \ p.\ 905]{Avellaneda1991}, there exists constants~$C>0$ and~$\beta \in (0,1)$ such that, for all~$x' \neq y' \in \mathbb{R}^d$,
	\begin{align}
	  \nonumber
	  & \Bigg|\nabla_x \nabla_y \mathcal{G}(x',y') 
	  - \sum_{i,j=1}^d  \partial_{x_{i}}\partial_{y_j} \mathcal{G}_\star(x'-y')\left(e_i + \nabla w_i\left(x'\right) \right) \otimes \left( e_j+\nabla w_j^\dagger\left( y'\right) \right) \Bigg|
	  \\
	  \label{Homog_Green_DefGreen3}
	  &\leq C |x'-y'|^{-d-\beta}.
	\end{align}
	In~\eqref{Homog_Green_DefGreen3}, $w_i$ and~$w_j^\dagger$ denote the correctors associated to the matrix~$A$, respectively~$A^T$.
	That is, $w_i$ is the periodic function of zero mean satisfying
	\begin{align*}
	  -{\rm div}\left(A(x)\cdot \left(\nabla w_i(x)+e_i\right) \right)=0, \quad\text{for}\quad x \in {\rm Q}.
	\end{align*}
	Identity~\eqref{Homog_Green_ExpressionDerive2} and Estimate~\eqref{Homog_Green_DefGreen3} imply
	\begin{align}\label{Homog_Green_SurH1}
	  H^k(x,y)
	  =&
	  H^{1,k}(x,y)+H^{2,k}(x,y),
	\end{align}
	where
	\begin{align}\label{Homog_Green_SurH2}
	  \left|H^{1,k}(x,y)\right| \leq C |k|^{-d-\beta},
	\end{align}
	and
	\begin{align*}
	  H^{2,k}(x,y):=&
	  \sum_{i, j=1}^d\int_{{\rm Q}^2} \int_{[0,1]^2} 
	  \left(x' \cdot \left(e_i + \nabla w_i(x+tx') \right)\right) 
	  \\
	  &~~~
	  \left(y' \cdot \left(e_j + \nabla w_j^\dagger(y+\tau y'-k) \right)\right)
	  \\
	  &~~~
	  \partial_{x_{i}}\partial_{y_j} \mathcal{G}_\star\left(x+tx'-(y+\tau y'-k)\right) 
	  {\rm{d}} \tau  {\rm{d}} t {\rm{d}} y' {\rm{d}} x'.
	\end{align*}
	
	All the correctors~$w_i$ and~$w_j^\dagger$ are bounded; furthermore, Formula~\eqref{Homog_Green_GreenLaplace} implies that the third-order derivatives of~$\mathcal{G}_\star$ evaluated at~$x-y$ are bounded by~$C |x-y|^{-d-1}$, where~$C$ is a constant independent of~$x$ and~$y$. 
	Therefore, a Taylor expansion yields a constant~$C$ such that
	\begin{align}
	  &\Bigg|H^{2,k}(x,y)-\sum_{i, j=1}^d\partial_{x_{i}}\partial_{y_j} \mathcal{G}_\star(k) 
	  Q_{ij}(x,y)\Bigg|
	  \leq C |k|^{-d-1},
	  \label{Homog_Green_SurH3}
	\end{align}
	where
	\begin{align*}
	  Q_{ij}(x,y):=&\int_{{\rm Q}^2} \int_{[0,1]^2} 
	  \left(x' \cdot \left(e_i + \nabla w_i(x+tx') \right)\right) 
	  \\&
	  \left(y' \cdot \left(e_j + \nabla w_j^\dagger(y+\tau y') \right)\right)
	  {\rm{d}} \tau  {\rm{d}} t {\rm{d}} y' {\rm{d}} x'.
	\end{align*}
	Then, a straightforward integration yields
	\begin{align*}
	  Q_{ij}(x,y)=&\int_{{\rm Q}}\left(x'_i + w_i(x+x')- w_i(x) \right) {\rm{d}} x'
	  \int_{{\rm Q}} \left(y'_j + w_j^\dagger(y+y') - w_j^\dagger(y)\right) {\rm{d}} y',
	\end{align*}
	which can be simplified as
	\begin{align*}
	  Q_{ij}(x,y)=& w_i(x)w_j^\dagger(y),
	\end{align*}
	since the correctors ~$w_i$ and~$w_j^\dagger$ are of zero mean.
	Note that~$Q_{ij}$ defined above does not depend on~$k \in \mathbb{Z}^d$ because the correctors~$w_i$ and~$w_j^\dagger$ are periodic.
	As a consequence, collecting~\eqref{Homog_Green_SurH1}, \eqref{Homog_Green_SurH2}, and~\eqref{Homog_Green_SurH3} yields
	\begin{align}
	  \left|\sum_{k \in \Gamma_m} H^k(x,y) \right|
	  \leq& C 2^{-m\beta}
	  + 
	  \left|Q_{ij}(x,y)\right| \sum_{i, j=1}^d
	  \left|\sum_{k \in \Gamma_m}
	  \partial_{x_{i}}\partial_{y_j} \mathcal{G}_\star(k) \right|.
	  \label{Homog_Green_SurH4}
	\end{align}
	
      \paragraph{Step 2:}
	Remark that~$Q_{ij}(x,y) \neq 0$ in general and that~$\left|\partial_{x_{i}}\partial_{y_j} \mathcal{G}_\star(k)\right|$ scales like~$|k|^{-d}$. 
	Therefore, by \eqref{Homog_Green_SurH3}, in general, the series in~\eqref{Homog_Green_Decomposition} does not converge  absolutely with respect to~$k$.
	
	Invoking once more~\eqref{Homog_Green_GreenLaplace}, we obtain
	\begin{align*}
	  \partial_{x_{i}}\partial_{y_j} \mathcal{G}_\star(k)
	  =
	  \left\{
	  \begin{aligned}
	    &C_d d (d-2) \frac{k_ik_j}{|k|^{d+2}} 
	    &&\quad\text{if}\quad i \neq j,\\
	    &C_d (d-2) \frac{dk_i^2-|k|^2}{|k|^{d+2}} 
	    &&\quad\text{if}\quad i = j.\\
	  \end{aligned}
	  \right.
	\end{align*}
	Thanks to the symmetry of~$\Gamma_m$ with respect to the hyperplane~$x_i=0$, in the case~$i \neq j$, and thanks to the invariance of~$\Gamma_m$ under the relabeling of the components of the vector~$k$, in the case~$i=j$, we deduce that
	\begin{align}\label{Homog_Green_PropGstar}
	  \sum_{k \in \Gamma_m} \partial_{x_{i}}\partial_{y_j} \mathcal{G}_\star(k) =0, \qquad \forall i, j \in [\![1,d]\!].
	\end{align}
	As a consequence, recalling~\eqref{Homog_Green_SurH4},
	\begin{align}
	  \label{Homog_Green_MajorH}
	  \left|\sum_{k \in \Gamma_m} H^k(x,y) \right| \leq  C 2^{-m\beta}.
	\end{align}
	
	Moreover, by~\cite[Prop.\ 4]{LegollGreen}, there exists a constant~$C$ such that for any~$x' \neq y'$, there holds
	\begin{align}
	  \label{Homog_Green_LegollGreen:1}
	  \left|\mathcal{G}(x',y')\right|\leq C |x'-y'|^{-d+2}.
	\end{align}
	Hence, for any~$k \in \mathbb{Z}^d$, 
	\begin{align}
	  \left|H^k(x,y)\right| \leq & 
	    C |x-y|^{-d+2}.
	    \label{Homog_Green_SurH5}
	\end{align}
	As a consequence of~\eqref{Homog_Green_MajorH} and~\eqref{Homog_Green_SurH5}, the series~\eqref{Homog_Green_Decomposition} converges absolutely in~$m$ as follows:
	\begin{align*}
	  \sum_{m=0}^{+\infty} \left|\sum_{k \in \Gamma_m} H^k(x,y)\right| \leq C |x-y|^{-d+2},
	\end{align*}
	for all~$x \neq y$, $y - x \in {\rm Q}$. Thus, we have recovered~\eqref{Homog_Green_Green1_debut} by an approach different from Proposition~\ref{Homog_Green_ThGreen1}.
    \end{proof}
    
    Now that we have justified that the series in~\eqref{Homog_Green_Decomposition} converges, we prove that  its limit, that we denote by~$\overline{G}$ for the moment, \textit{id est},
      \begin{align}
	\label{Homog_Green_DefGnT}
	\overline{G}(x,y):=\sum_{m=0}^{+\infty} \left( \sum_{k \in \Gamma_m} H^k(x,y) \right),
      \end{align}
      is actually equal to ~$G$.
    It can easily be checked that $\overline{G}$ is periodic in~$x$ and~$y$ and satisfies~\eqref{Homog_Green_DefG2}. The technical point is~\eqref{Homog_Green_MoyenneGreen}, the proof of which relies on the former Taylor expansion and on the classical result~\cite[Prop.\ 8]{LegollGreen}. According to the latter, there exists a constant~$C>0$ such that, for all~$x' \neq y'$,
    \begin{align}\label{Homog_Green_LegollGreen:3}
      \left|\nabla_x \nabla_y \mathcal{G}(x',y')\right| \leq C |x'-y'|^{-d}.
    \end{align}

    \begin{proof}[Proof of Identity~\eqref{Homog_Green_Decomposition}]
      Obviously, $\overline{G}$ is periodic in~$y$. Since~$A$ is periodic, there even holds
      \begin{align*}
	\mathcal{G}(x,y-k)=\mathcal{G}(x+k,y),
      \end{align*}
      for any~$k \in \mathbb{Z}^d$.
      Therefore~$\overline{G}$ is also periodic in~$x$.
      Moreover, we check that
      \begin{align*}
	-{\rm div}\left(A(x) \cdot \nabla H^k(x,y) \right)=&
	\delta_0(x-y+k)
	\\
	&
	-2\chi_{{\rm Q}}(x-y+k)
	+\int_{{\rm Q}} \chi_{{\rm Q}}(x-y-y'+k) {\rm{d}} y',
      \end{align*}
      where $\chi_{{\rm Q}}$ is the characteristic function of the set ${\rm Q}$.
      Whence
      \begin{align*}
	  -{\rm div}\left(A(x) \cdot \left(\sum_{m=0}^{+\infty} 
	  \left(\sum_{k \in \Gamma_m} 
	  \nabla_x H^k(x,y)\right)\right)\right)
	  =\sum_{k \in \mathbb{Z}^d}\delta_0(x-y+k) - 1 \quad\text{in}\quad \mathbb{R}^d.
      \end{align*}
      To summarize, $\overline{G}(x,y)$ defined by~\eqref{Homog_Green_DefGnT} is~$x$-periodic and~$y$-periodic, and satisfies~\eqref{Homog_Green_DefG2}.
      
      Next, we justify that ~$\overline{G}$ satisfies~\eqref{Homog_Green_MoyenneGreen}.
      By integrating~\eqref{Homog_Green_ExpressionDerive2} along the~$y$ variable, there holds
      \begin{align*}
        \int_Q H^k(x,y) {\rm{d}} y = &
        \int_0^1\int_0^1 \int_{{\rm Q}}\int_{y \in \partial Q}    \\
        &~~~x' \cdot \Bigg( \int_{{\rm Q}}
         \nabla_x \mathcal{G}(x+tx',y+\tau y'-k)  \otimes y' {\rm{d}} y' \Bigg) \cdot {\rm{d}} \vec{S}(y)  {\rm{d}} x' {\rm{d}} \tau{\rm{d}} t.
      \end{align*}
      Hence, due to cancellations on the boundaries of the translated cubes $k +Q$,
      \begin{align*}
	\int_{{\rm Q}} \sum_{|k| < 2^m} H^k(x,y) {\rm{d}} y
	=&
	\int_0^1\int_0^1 \int_{{\rm Q}}\int_{y \in \Xi_m} x' \cdot \Bigg( \int_{{\rm Q}} 
         \nabla_x \mathcal{G}(x+tx',y+\tau y'-k) 
	\\
	&~~~
         \otimes  y'  {\rm{d}} y' \Bigg) \cdot  {\rm{d}} \vec{S}(y) {\rm{d}} x' {\rm{d}} \tau{\rm{d}} t,
      \end{align*}
      where~$\Xi_m$ is the boundary of the following set:
      \begin{align*}
        \underset{|k|< 2^{m+1}}{\cup} \left(k+Q\right).
      \end{align*}
      Now, by Taylor expansion, and thanks to~\eqref{Homog_Green_LegollGreen:3}, for all~$\tau \in [0,1]$, $x, x', y' \in {\rm Q}$, and $y \in \Xi_m$,
	\begin{align*}
	  \left|\nabla_x \mathcal{G}(x+tx',y+\tau y')-\nabla_x \mathcal{G}(x+tx',y)\right| \leq C 2^{-md}.
	\end{align*}
	Therefore
	\begin{align*}
	  \left|\int_{{\rm Q}}\nabla_x \mathcal{G}(x+tx',y+\tau y') \otimes y' {\rm{d}} y'\right|
	  \leq \left|\int_{{\rm Q}}\nabla_x \mathcal{G}(x+tx',y) \otimes y' {\rm{d}} y'\right|+C2^{-md}.
	\end{align*}
	The integral in the right-hand term of the above estimate vanishes since, by symmetry,
	  $\int_{{\rm Q}}  y' {\rm{d}} y'=0$.
	As a consequence, 
	\begin{align*}
	  \left|\int_{{\rm Q}}\nabla_x \mathcal{G}(x+tx',y+\tau y') \otimes y' {\rm{d}} y'\right|
	  \leq C2^{-md}.
	\end{align*}
	Whence, since the surface area of $\Xi_m$ is bounded by $C 2^{m(d-1)}$, we have
	\begin{align*}
	  \left|\int_{{\rm Q}} \sum_{|k| < 2^{m+1}} H_{k}(x,y) {\rm{d}} y\right|
	  \leq& C2^{-m}.
	\end{align*}
	As a consequence, letting~$m \rightarrow+\infty$, we deduce that~$\overline{G}$ satisfies~\eqref{Homog_Green_MoyenneGreen}.
	By the same arguments transposed from~$\overline{G}(x,y)$ to~$\overline{G}(y,x)$, it can be shown that~$\overline{G}$ also satisfies~\eqref{Homog_Green_MoyenneGreen2}.
	Therefore, by Proposition~\ref{Homog_Green_PropExistUne}, we have
	\begin{align*}
	  \overline{G}(x,y)=G_n(x,y),
	\end{align*}
	which concludes the proof.
      \end{proof}

    \subsection{General case}\label{Homog_Green_Sym_Mat}
    
      As is easily seen in the proof of Proposition~\ref{Homog_Green_ThDecompose}, the fact that the homogenized matrix~$A^\star$ is the identity is only used for establishing~\eqref{Homog_Green_PropGstar}.
      One also realizes that, would~\eqref{Homog_Green_PropGstar} be replaced by the following estimates:
      \begin{align}\label{Homog_Green_PropGstar2}
        \left|\sum_{k \in \Gamma_m} \partial_{x_{i}}\partial_{y_j} \mathcal{G}_\star(k) \right| \leq C_m \text{ for all } m \in [\![1,+\infty ),  \quad\text{and}\quad \sum_{m= 1}^{+\infty} C_m<+\infty,
      \end{align}
      for well-chosen sets $\Gamma_m$, then the conclusions of Proposition~\ref{Homog_Green_ThDecompose} would also apply.

      We show that the sets $\Gamma_m$ defined by~\eqref{Homog_Green_DefSet1} and~\eqref{Homog_Green_DefSet2} are such that Estimates~\eqref{Homog_Green_PropGstar2} are satisfied.
      Hence, the conclusions of Proposition~\ref{Homog_Green_ThDecompose} are true without any assumption on the homogenized matrix~$A^\star$ of $A$.

      The homogenized matrix~$A^\star$ is (constant) coercive.
      Henceforth the Green function in~$\mathbb{R}^d$ associated with the operator~$-{\rm div}\left(A^\star\cdot \nabla \right)$ is
      \begin{align*}
        \mathcal{G}_\star(x)=\frac{C\left(A^\star_{\rm s}\right)}{\left(x\cdot \left(A^\star_{\rm s}\right)^{-1} \cdot x\right)^{(d-2)/2}},
      \end{align*}
      where~$C\left(A^\star_{\rm s}\right)$ is a constant and~$A^\star_{\rm s}$ is the symmetric part of the matrix~$A^\star$.
      Whence
      \begin{align*}
         \nabla^2\mathcal{G}_\star(x)=C(d-2) \frac{d \left(\left(A^\star_{\rm s}\right)^{-1} \cdot x \right)\otimes \left(\left(A^\star_{\rm s}\right)^{-1} \cdot x\right)-\left(x \cdot \left(A^\star_{\rm s}\right)^{-1} \cdot x \right) \left(A^\star_{\rm s}\right)^{-1}}{\left(x\cdot \left(A^\star_{\rm s}\right)^{-1} \cdot x\right)^{(d+2)/2}}.
      \end{align*}

      Besides, there exists an orthogonal matrix~$O$ and positive scalars~$\lambda_i$, $i \in [\![1,d]\!]$ such that
      \begin{align*}
        A^\star_{\rm s} = O^{-1} \cdot \text{diag}\left(\lambda_1^{-2},\cdots,\lambda_d^{-2}\right) \cdot O.
      \end{align*}
      Therefore, denoting by~$f_j$ the orthonormal base related to~$O$, and decomposing
      \begin{align*}
        x=\sum_{j=1}^d \lambda_j^{-1}\widetilde{x}_j f_j,
      \end{align*}
      one obtains
      \begin{align*}
        \nabla^2\mathcal{G}_\star(x)=C(d-2) \frac{d \sum_{i, j=1}^d \lambda_i\lambda_j \widetilde{x}_i \widetilde{x}_j f_i \otimes f_j-\left(\sum_{i=1}^d \widetilde{x}_i^2 \right) \left(\sum_{i=1}^d \lambda_i^2 f_i \otimes f_i\right)}{\left(\sum_{i=1}^d \widetilde{x}_i^2\right)^{(d+2)/2}}.
      \end{align*}
      For~$r \in \mathbb{R}_+$, we define the following set:
      \begin{align}
        \label{Homog_Green_DefOmegam}
        \Omega_r:=\left\{x \in \mathbb{R}^d, \sum_{j=1}^d |\widetilde{x}_j|^2=r^2 \right\}.
      \end{align}
      Remark that there obviously holds
      \begin{align*}
	\Omega_r= \left\{ x \in \mathbb{R}^d, \left(x\cdot \left(A^\star_{\rm s}\right)^{-1} \cdot x\right)=r^2\right\}.
      \end{align*}
      On the one hand, if~$i \neq j$, $f_i\otimes f_j : \nabla^2\mathcal{G}_\star(x)$ changes sign with respect to the transformation~$\widetilde{x}_j \mapsto -\widetilde{x}_j$.
      Therefore, if~$i \neq j$, there holds
      \begin{align}\label{Homog_Green_Gstar1}
        \int_{\Omega_r} f_i\otimes f_j : \nabla^2\mathcal{G}_\star(x){\rm{d}} S(x)=0.
      \end{align}
      On the other hand,
      \begin{align*}
         f_i\otimes f_i:\nabla^2\mathcal{G}_\star(x)
        =C(d-2) \frac{\lambda_i^2 \left(d \widetilde{x}_i^2 -\left(\sum_{j=1}^d \widetilde{x}_j^2 \right)\right)}
        {\left(\sum_{k=1}^d \widetilde{x}_k^2\right)^{(d+2)/2}}.
      \end{align*}
      By invariance of~$\Omega_r$ under the relabeling of the coordinates~$\widetilde{x}_j$,
      \begin{align*}
        \int_{\Omega_r} \frac{\sum_{k=1}^d \widetilde{x}_k^2}
        {\left(\sum_{k=1}^d \widetilde{x}_k^2\right)^{(d+2)/2}}{\rm{d}} S(x)
        =d \int_{\Omega_r} \frac{\widetilde{x}_i^2}{\left(\sum_{k=1}^d \widetilde{x}_k^2\right)^{(d+2)/2}}{\rm{d}} S(x).
      \end{align*}
      As a consequence,
      \begin{align}
        \int_{\Omega_r} f_i \otimes f_i : \nabla^2\mathcal{G}_\star(x){\rm{d}} S(x)=0.
        \label{Homog_Green_Gstar2}
      \end{align}
      
      Hence, we define
      \begin{align*}
        \Lambda_m:=\underset{r \in [2^m,2^{m+1})}{\cup} \Omega_r,
      \end{align*}
      and~$\Gamma_m:=\Lambda_m\cap \mathbb{Z}^d$.
      By convergence of Riemann integrals, we deduce that
      \begin{align*}
        \left|\sum_{k \in \Gamma_m} \partial_i\partial_j \mathcal{G}_\star(k) \right|
        \leq& C
        \left|\partial \Lambda_m\right|\sup_{x \in \Lambda_m} \left|\partial_i\partial_j \mathcal{G}_\star(x) \right|
        \\
        &+C\left| \Lambda_m\right| \sup_{x \in \Lambda_m} \left|\partial_i\partial_j \nabla \mathcal{G}_\star(x) \right|
        \\
        &+C\left| \int_{\Lambda_m} \partial_i\partial_j\mathcal{G}_\star(x) {\rm{d}} x \right|.
      \end{align*}
      By~\eqref{Homog_Green_Gstar1} and~\eqref{Homog_Green_Gstar2}, we deduce that the last integral in the above estimate vanishes.
      Furthermore, straightforward estimates on the derivatives of~$\mathcal{G}_\star$ yield
      \begin{align*}
        \left|\sum_{k \in \Gamma_m} \partial_i\partial_j \mathcal{G}_\star(k) \right|
        \leq& C 2^{-m}.
      \end{align*}
      This implies the convergence of the series in~\eqref{Homog_Green_Decomposition} in the general case.

    \subsection{Case of systems}\label{Homog_Green_Systems}
      
      In the case of systems, the Green function~$\mathcal{G}_\star$ of~$-{\rm div}\left(A^\star \cdot \nabla \right)$ in~$\mathbb{R}^d$ reads:
      \begin{align}\label{Homog_Green_FuncSys}
        \mathcal{G}_\star^{\alpha\beta}(x)= C\left(\left(A^\star_{\rm s}\right)^{\alpha\beta}\right) \left( \sum_{i, j=1}^d x_i \left(\left(A^\star_{\rm s}\right)^{\alpha\beta}\right)^{-1}_{ij}x_j \right)^{-\frac{d-2}{2}},
      \end{align}
      where~$C\left(\left(A^\star_{\rm s}\right)^{\alpha\beta}\right)$ is a constant depending on the matrix~$\left(\left(A^\star_{\rm s}\right)^{\alpha\beta}_{ij}\right)_{i, j}$, and
      \begin{align*}
        \left(A^\star_{\rm s}\right)^{\alpha\beta}_{ij} = \frac{1}{2}\left( \left(A^\star\right)^{\alpha\beta}_{ij}+\left(A^\star\right)^{\alpha\beta}_{ji}\right).
      \end{align*}
      Whence, by the above arguments of Sections~\ref{Homog_Green_SecId} and~\ref{Homog_Green_Sym_Mat}, we have the following decomposition:
      \begin{align}\label{Homog_Green_Decomposition_Sys}
	G^{\alpha\beta}(x,y)
	=\sum_{m=0}^{+\infty} 
	\left(\sum_{k \in \Gamma^{\alpha\beta}_m} 
	\left(H^k\right)^{\alpha\beta}(x,y) \right),
      \end{align}
      where the definition of each term will be made precise below.
      
      The functions~$H^k$ are defined by 
      \begin{align}
	\nonumber
	\left(H^k\right)^{\alpha\beta}(x,y)
	&:=
	\mathcal{G}^{\alpha\beta}(x,y-k)
	-\int_{{\rm Q}} \mathcal{G}^{\alpha\beta}(x,y+y'-k) {\rm{d}} y'
	\\
	\nonumber
	&~~~~-\int_{{\rm Q}} \mathcal{G}^{\alpha\beta}(x+x',y-k) {\rm{d}} x'
	\\
	&~~~~+\int_{{\rm Q}} \int_{{\rm Q}} \mathcal{G}^{\alpha\beta}(x+x',y+y'-k) {\rm{d}} y' {\rm{d}} x',
	\label{Homog_Green_DefHk_Sys}
      \end{align}
      where the function $\mathcal{G}$ is the Green function in~$\mathbb{R}^d$ of the operator~$-{\rm div}\left(A \cdot \nabla\right)$.
      Last, we define the sets~$\Gamma_m^{\alpha\beta}$ by:
      \begin{align*}
        &\Gamma_0^{\alpha\beta}=\left\{k \in \mathbb{Z}^d, 0 \leq k\cdot \left(\left(A^\star_{\rm s}\right)^{\alpha\beta}\right)^{-1} \cdot k < 2^2 \right\}, \\
        &\Gamma_m^{\alpha\beta} =\left\{k \in \mathbb{Z}^d, 2^{2m} \leq k\cdot \left(\left(A^\star_{\rm s}\right)^{\alpha\beta}\right)^{-1} \cdot k < 2^{2m+2} \right\} && \quad\text{if}\quad m \geq 1,
      \end{align*}
      where $\left(A^\star_{\rm s}\right)^{\alpha\beta}$ is considered as a matrix in~$\mathbb{R}^{d\times d}$.

    \section*{Acknowledgement}
      The author would like to thank both Fr\'ed\'eric Legoll and Pierre-Lo\"ic Roth\'e, who have brought his attention to this question.
      Fr\'ed\'eric Legoll has also helped for the proof in the $2$-dimensional case.
      Moreover, the author acknowledges Claude Le Bris for his suggestions concerning the decomposition of the Green functions.
      Finally, the author gratefully thanks Xavier Blanc for reading the first version of this article.
    
    \bibliographystyle{plain}

\begin{thebibliography}{10}

\bibitem{Anantharaman_2012}
A.~Anantharaman, R.~Costaouec, C.~Le~Bris, F.~Legoll, and F.~Thomines.
\newblock Introduction to numerical stochastic homogenization and the related
  computational challenges: some recent developments.
\newblock In {\em Multiscale modeling and analysis for materials simulation},
  volume~22 of {\em Lect. Notes Ser. Inst. Math. Sci. Natl. Univ. Singap.},
  pages 197--272. World Sci. Publ., Hackensack, NJ, 2012.

\bibitem{AvellanedaLin}
M.~Avellaneda and F.-H. Lin.
\newblock Compactness methods in the theory of homogenization.
\newblock {\em Comm. Pure Appl. Math.}, 40(6):803--847, 1987.

\bibitem{Avellaneda1991}
M.~Avellaneda and F.-H. Lin.
\newblock {${\rm L}^p$} bounds on singular integrals in homogenization.
\newblock {\em Comm. Pure Appl. Math.}, 44(8-9):897--910, 1991.

\bibitem{BerghLofstrom}
J.~Bergh and J.~L\"ofstr\"om.
\newblock {\em Interpolation spaces. {A}n introduction}.
\newblock Springer-Verlag, Berlin-New York, 1976.

\bibitem{LegollGreen}
X.~Blanc, F.~Legoll, and A.~Anantharaman.
\newblock Asymptotic behavior of {G}reen functions of divergence form operators
  with periodic coefficients.
\newblock {\em Appl. Math. Res. Express. AMRX}, (1):79--101, 2013.

\bibitem{Brezis}
H.~Brezis.
\newblock {\em Functional analysis, {S}obolev spaces and partial differential
  equations}.
\newblock Universitext. Springer, New York, 2011.

\bibitem{Le_Bris_GreenPer}
I.~Catto, C.~Le~Bris, and P.-L. Lions.
\newblock {\em The mathematical theory of thermodynamic limits:
  {T}homas-{F}ermi type models}.
\newblock Oxford Mathematical Monographs. The Clarendon Press, Oxford
  University Press, New York, 1998.

\bibitem{Dolzmann}
G.~Dolzmann and S.~M{\"u}ller.
\newblock Estimates for {G}reen's matrices of elliptic systems by {$L^p$}
  theory.
\newblock {\em Manuscripta Math.}, 88(2):261--273, 1995.

\bibitem{GT}
D.~Gilbarg and N.~Trudinger.
\newblock {\em Elliptic partial differential equations of second order}.
\newblock Classics in Mathematics. Springer-Verlag, Berlin, 2001.

\bibitem{GruterWidman}
M.~Gr\"uter and K.-O. Widman.
\newblock The {G}reen function for uniformly elliptic equations.
\newblock {\em Manuscripta Math.}, 37(3):303--342, 1982.

\bibitem{PhD_Marc}
M.~Josien.
\newblock Th\`ese de l'{U}niversit\'e {P}aris {E}st: Etude math\'ematique et
  num\'erique de quelques mod\`eles multi-\'echelles issus de la m\'ecanique
  des mat\'eriaux, 2018.
\newblock In preparation.

\bibitem{KLS_2013}
C.~Kenig, F.-H. Lin, and Z.~Shen.
\newblock Homogenization of elliptic systems with {N}eumann boundary
  conditions.
\newblock {\em J. Amer. Math. Soc.}, 26(4):901--937, 2013.

\bibitem{KLSGreenNeumann}
C.~Kenig, F.-H. Lin, and Z.~Shen.
\newblock Periodic homogenization of {G}reen and {N}eumann functions.
\newblock {\em Comm. Pure Appl. Math.}, 67(8):1219--1262, 2014.

\bibitem{Minvielle_2015}
F.~Legoll and W.~Minvielle.
\newblock Variance reduction using antithetic variables for a nonlinear convex
  stochastic homogenization problem.
\newblock {\em Discrete Contin. Dyn. Syst. Ser. S}, 8(1):1--27, 2015.

\bibitem{LegollRothe}
F.~Legoll and P.-L. Roth\'e.
\newblock On the numerical approximation of fluctuations in stochastic
  homogenization.
\newblock in preparation.

\end{thebibliography}
    \def\cprime{$'$} \def\cprime{$'$}

\end{document}